\documentclass[12pt]{amsart}
\usepackage{amsfonts}
\usepackage{amsmath}
\usepackage{amsthm, amscd}
\usepackage{mathtools}
\usepackage{mathabx}
\usepackage{amssymb}
\usepackage{amsrefs}
\usepackage{mathrsfs}
\usepackage{color}
\usepackage[active]{srcltx}

\theoremstyle{plain}
\newtheorem{thm}{Theorem}
\newtheorem{cor}[thm]{Corollary}
\newtheorem{lem}[thm]{Lemma}

\newtheorem{prop}[thm]{Proposition}
\newtheorem*{Theorem*}{thm}
\theoremstyle{remark}
\newtheorem{rmk}[thm]{Remark}
\newtheorem*{rmk*}{Remark}

\DeclareMathOperator{\Tor}{Tor}

\DeclareMathOperator{\UU}{U}
\DeclareMathOperator{\pt}{pt}
\DeclareMathOperator{\ind}{ind}

\newcommand{\frs}{\mathfrak{s}}

\newcommand{\HMbar}{\overline{\text{HM}}}
\newcommand{\HFinfty}{\text{HF}^\infty}
\newcommand{\HPbar}{\overline{\text{HP}}}

\begin{document}

\title[$U$--cyclic elements in Floer homology]{A note on $U$--cyclic elements in monopole Floer homology}

\begin{abstract}

Edtmair--Hutchings have recently defined, using Periodic Floer homology (PFH), a ``$U$--cycle property'' for Hamiltonian isotopy classes of area-preserving diffeomorphisms of closed surfaces. They show that every Hamiltonian isotopy class satisfying the $U$--cycle property satisfies the smooth closing lemma and also satisfies a kind of Weyl law involving the actions of certain periodic points; they show that every rational isotopy class on the two-torus satisfies the $U$--cycle property.  It seems that, in general, not much is known about the U-module structure on PFH.  Here we consider a version of Seiberg-Witten Floer cohomology which is known by work of Lee-Taubes to be isomorphic, as a U-module, to the periodic Floer homology in sufficiently high degree.  We show that the analogous $U$--cycle property holds for every rational Hamiltonian isotopy class on any closed surface and, more generally, for any non-torsion spin-c structure. On the other hand, we also show that a rational isotopy class may contain elements that are not $U$--cyclic.  By the Lee-Taubes isomorphism, the same results hold for PFH.  Our results are some of the first computations concerning the U-module structure on these theories.

\end{abstract}

\author[Cristofaro-Gardiner]{Dan Cristofaro-Gardiner}
\address{University of Maryland\\College Park, MD 20742}
\author[Pomerleano]{Daniel Pomerleano}
\address{University of Massachusetts\\Boston, MA 02125}
\author[Prasad]{Rohil Prasad}
\address{Princeton University\\Princeton, NJ 08544}
\author[Zhang]{Boyu Zhang}
\address{Princeton University\\Princeton, NJ 08544}

\date{}

\maketitle

\section{Introduction}

Let $Y$ be a closed three-manifold.  Let $\frs$ be a spin-c structure over $Y$ such that $c_1(\frs)$ is not a torsion element. Kronheimer--Mrowka \cite[Section 29]{monopolesBook} have defined various versions of monopole Floer homology with non-exact perturbations. To explain the variant of interest to us here, recall from \cite{monopolesBook} that every non-exact perturbation has a \emph{period class} $c\in H^2(Y;\mathbb{R})$. 
A perturbation is called \emph{monotone} if its period class $c$ satisfies 
$$2\pi^2c_1(\frs) + c = t\,2\pi^2c_1(\frs) \in H^2(Y;\mathbb{R})$$
for some $t\in \mathbb{R}$. 
 It is called \emph{positively monotone}, \emph{negatively monotone}, or \emph{balanced} if $t>0$, $t<0$, or $t=0$ respectively (see \cite[Definition 29.1.1]{monopolesBook}).  

Here we are interested in the ``bar-version" $\HMbar^{*}(Y,\frs,c_b)$ of monopole Floer cohomology with balanced perturbation as defined in \cite[Section 30.1]{monopolesBook}. Let $2N$ be the divisibility of $c_1(\frs)$, then the group $\HMbar^{*}(Y,\frs,c_b)$ is equipped with a relative $\mathbb{Z}/2N$--grading, and there is a U-map defined on $\HMbar^{*}(Y,\frs,c_b)$ with relative degree $+2$ (see \cite[Section 25]{monopolesBook}).  The work of Eismeier-Lin provides a computation of $\HMbar^{*}$ in this case  in terms of the ``extended cup homology" of $Y$, see \cite[Thm. 6.4]{lin2021monopole}; however, this does not give a complete picture of the $U$-module structure directly, and this is the question of interest to us here.  

One motivation for studying the $U$-map structure is a recent result of Edtmair--Hutchings, leveraging ``$U$--cyclic classes", which we now introduce. Suppose $\Sigma$ is a closed oriented surface endowed with an area form $\omega_\Sigma$, suppose $\phi$ is an area-preserving diffeomorphism of $\Sigma$, and fix a class $\Gamma \in H_1(Y_\phi)$, where $M_{\phi}$ denotes the corresponding mapping torus.  Assume that $\Gamma$ is {\em monotone}, meaning that $c_1(V) + 2 PD(\Gamma)$ is in the span of $[\omega]$, where $V$ denotes the vertical tangent bundle and $\omega$ is the two-form induced from the area form on $\Sigma$.  Then the {\em periodic Floer homology} $\overline{HP}(\phi,\Gamma;R)$ is defined.  It is the homology of a chain complex generated by certain sets of Reeb orbits, with coefficients in $R$, relative to a differential counting certain pseudoholomorphic curves.  It has a distinguished map with relative degree $-2$, defined by counting certain index two psuedoholomorphic curves, called the $U$-map.  We now say that an element $\sigma \in \overline{HP}(\phi,\Gamma;R)$ is {\em $U$--cyclic} if $U^k \sigma = \sigma$ for some positive integer $k$, and we say that 
$\phi$ {\em satisfies the $U$--cycle property}\footnote{Strictly speaking, the definition in Edtmair--Hutchings is slightly less restrictive than this: they have a family of variants of PFH, indexed by subgroups $G \subset \text{Ker}([\omega_\phi])$, and require just one of these variants to have $U$--cyclic classes.  Our definition here is equivalent to the case $G = \text{Ker}([\omega_\phi]),$ by \cite[Lem. 2.23]{edtmairHutchings}.} if there are $U$--cyclic elements in $\overline{HP}(\phi,\Gamma;\mathbb{Z}/2)$ for $\Gamma \in H_1(M_\phi)$ of arbitrarily large degree.  We will not need to know much more about PFH in the present work, but for the reader curious to know more we refer to the discussions in, for example, \cite{cristofaro2021smooth,edtmairHutchings}.  

The work of Edtmair and Hutchings \cite{edtmairHutchings} proves important properties of $U$--cyclic elements.  
They proved a Weyl asymptotic formula for the PFH spectral invariants associated to $U$--cyclic classes; for brevity, we will not recall the precise statement here, referring the reader to \cite{edtmairHutchings} instead.  They also prove\footnote{Independently and simultaneously, a Weyl law for the PFH spectral invariants, without any $U$--cyclic condition, was proved in \cite{cristofaro2021smooth} for all rational isotopy classes of area-preserving maps, and the closing lemma for general area-preserving diffeomorphisms was proved.} various closing lemmas for any Hamiltonian isotopy class satisfying the $U$--cycle property; closing lemmas, in particular, have long attracted considerable interest and we refer the reader to Remark~\ref{rmk:summary} below for further discussion.  This leads to a natural question: how plentiful are $U$--cyclic elements? 

By an isomorphism of Lee-Taubes \cite[Theorem 1.2, Corollary 1.5]{LeeTaubes12}, when $\Gamma$ is monotone with degree greater than $\max\{2g-2,0\}$, where $g$ is the genus of $\Sigma$, we have
\begin{equation}
\label{eqn_iso_HP_HMbar}
	\HPbar_*(\phi,\Gamma;\mathbb{Z})\cong \HMbar^{-*}(M_\phi,\mathfrak{s}_\Gamma,c_b;\mathbb{Z}),
\end{equation}
with all of the above isomorphisms preserving the U-action.  By Universal Coefficients, it follows from this that the same result holds with any coefficient ring $R$.  Here, $\mathfrak{s}_\Gamma$ denotes the spin-c structure whose first Chern class is $c_1(V) + 2 \text{PD}(\Gamma).$   In particular, one can understand the $U$-module structure on PFH by studying the corresponding structure on monopole Floer.  

A better understanding of the $U$-map is also the subject of Question 3 in the aforementioned Eismeier-Lin work \cite{lin2021monopole}.  

With these preliminaries and motivations now behind us, let us now return to discussing the $U$-module structure on $\HMbar^{-*}(Y,\mathfrak{s}_\Gamma,c_b;R)$.    It would be interesting to have a comprehensive understanding of this structure and this indeed seems a worthy question for further study.  We content ourselves in this short note here with some first results that are natural to wonder about given the Edtmair--Hutchings work explained above.

To state our first result, in analogy with above, we define an element $\sigma$ of $\HMbar^{-*}(Y,\mathfrak{s},c_b;R)$ to be {\em $U$--cyclic} if there exists $k\in\mathbb{Z}^+$ such that $U^k \sigma = \sigma$.  

\begin{thm}
\label{thm_U_cycle}
	For any $(Y,\mathfrak{s})$ with $\mathfrak{s}$ non-torsion, and coefficient ring $R$, the $U$-module 
	\[\HMbar^{*}(Y,\mathfrak{s}_\Gamma,c_b;R)\]
	always contains non-zero $U$--cyclic classes. 
\end{thm}

In fact, we give two proofs of this result.  The proof that we are mainly interested in here is  a direct and relatively short proof, building on work of Kronheimer--Mrowka, that stays entirely within the realm of Seiberg-Witten Floer theory.  One can alternatively give a very short but extremely indirect proof of the above theorem as a consequence of the nonvanishing theorem in \cite{cristofaro2021smooth}, together with a known computation of Heegaard Floer homology and the Kutluhan--Lee--Taubes isomorphism between monopole Floer homology and Heegaard Floer homology proved in \cite{kutluhan2020hf} --- we give the  argument in Section \ref{subsec_U-cycle_property} below, for completeness, as well.   
Since a rational Hamiltonian isotopy class gives rise to monotone classes $\Gamma$ of arbitrarily high degree (see e.g. \cite[Section 1.1.2]{cristofaro2021smooth}), we obtain in particular the following corollary by the Lee--Taubes isomorphism \eqref{eqn_iso_HP_HMbar}.

\begin{cor}
\label{cor_U_cycle}
	The $U$--cycle property holds for every rational Hamiltonian isotopy class of area-preserving diffeomorphisms on $\Sigma$.
\end{cor}

\begin{rmk}
\label{rmk:summary}
For the convenience of readers who are not specialists, we briefly summarize some of the background regarding closing lemmas.  The basic question here (stated, for expository simplicity, in the context of conservative dynamics) is as follows: for a generic volume-preserving diffeomorphism of a closed manifold, is the union of periodic points dense?  In the $C^1$--topology, this statement was proved by Pugh and Robinson \cite{pugh1983c1}.  In general, proving similar statements in higher regularity has been a long-standing open problem, and it is the subject (without the assumption of conservative dynamics) of Problem \#10 in Smale's  problem list \cite{smaleProblemList}.  

As mentioned above, the Edtmair--Hutchings result implies a resolution in the case of area-preserving maps of surfaces, under the $U$--cycle assumption that is resolved by Corollary~\ref{cor_U_cycle}; for a major antecedent breakthrough on this problem, see \cite{asaokairie}.  For maps in rational isotopy classes, Edtmair--Hutchings also prove some interesting quantitative results, estimating the time it takes for periodic orbits of a particular period to appear under a Hamiltonian isotopy in a fixed open set, see for example \cite{edtmairHutchings}[Thm. 7.4] for the most general statement of this form.  While, as mentioned above, \cite{cristofaro2021smooth} provides an alternate approach to addressing the closing lemma, it does not address these quantitative results. The constants in Edtmair--Hutchings' quantitative results depend on the minimum degree needed to get $U$--cyclic elements. Equation \eqref{eqn_iso_HP_HMbar} and Theorem \ref{thm_U_cycle} imply that $U$--cyclic elements exist whenever $\Gamma$ is monotone and the degree of $\Gamma$ is greater than $\max\{2g-2,0\}$.

For more about closing lemmas, we refer the reader to the discussions in \cite{cristofaro2021smooth, edtmairHutchings}, as well as the very useful notes for the Bourbaki seminar \cite{humiliere2019lemme}.  
\end{rmk}

Given Theorem~\ref{thm_U_cycle}, it is natural to ask if {\em every} class is $U$--cyclic, at least in high enough degree; this is also a natural question on the PFH side given the spectral asymptotics associated to $U$--cyclic classes in the work of Edtmair--Hutchings as explained above. Our second theorem answers this question in the negative.

\begin{thm}
\label{not_U_cycle}

Suppose $Y$ is a closed oriented three-manifold with $b_1(Y) = 3$. Let $a_1,a_2,a_3$ be generators of $H^1(Y;\mathbb{R})$, and assume 
$$a_1\cup a_2\cup a_3 \neq 0 \in H^3(Y;\mathbb{R}).$$   Let $\frs$ be a non-torsion spin-c structure over $Y$.  Then there exist infinitely many classes $\sigma \in \HMbar^{*}(Y,\frs,c_b;\mathbb{Z})$ that are not $U$--cyclic.
\end{thm}

Regarding PFH, we obtain the following corollary, which gives a simple example where there is an abundance of non $U$--cyclic classes.

\begin{cor}
\label{no_U_cycle}
	Let $\phi$ be any area-preserving diffeomorphism of $T^2$ that is Hamiltonian isotopic to the identity, and let $\Gamma \in H_1(M_\phi)$ be monotone, of positive degree.  Then there are classes in $\overline{HP}(\phi,\Gamma;\mathbb{Z})$ which are not $U$--cyclic.
\end{cor}

The above Theorem is certainly sensitive to the choice of coefficient ring $R$.  Indeed, our aforementioned proof of  Theorem~\ref{thm_U_cycle} shows that
\[ (1-U^{d-g + 1})^{b_1(Y)+1} \HMbar^{*}(M_{\phi},\mathfrak{s}_{\Gamma},c_b;R) = 0.\]
It was first pointed out to us by Edtmair \cite{edtmair} that this implies that every class is $U$--cyclic when $R = \mathbb{Z}/2$; a similar argument gives the same conclusion for some other rings, for example $R = \mathbb{Z}/p.$

\begin{rmk}

It seems interesting to us to prove  Corollary~\ref{no_U_cycle}, or a more general version, directly, that is without using the isomorphism \eqref{eqn_iso_HP_HMbar}; we refer the reader to the work of Nelson--Weiler \cite{nelson2020embedded} for some direct computations regarding the related embedded contact homology (ECH) in interesting cases.

\end{rmk}

\textbf{Acknowledgements: } We thank Oliver Edtmair for helpful correspondence concerning an earlier version of this draft.  D.C-G. thanks the National Science Foundation for their support under Awards \#1711976 and \#2105471. D. P.  is partially supported by the Simons Collaboration in Homological Mirror Symmetry, Award \#652299.  R.P. is supported by the NSF (Graduate Research Fellowship) under Award \#DGE--1656466. B.Z. would like to thank Mike Miller Eismeier, Francesco Lin, and Jianfeng Lin for helpful discussions.
 
\section{Proofs}

We now prove our two theorems.

\subsection{The $U$--cycle property}
\label{subsec_U-cycle_property}
Throughout this section, let $Y$ be a closed oriented 3-manifold and let $\mathfrak{s}$ be a spin-c structure over $Y$ such that $c_1(\mathfrak{s})$ is not torsion. Let $2N$ be the divisibility of $c_1(\mathfrak{s})$. 
Let $R$ be a commutative ring with $1\neq 0$.

In this section, we prove Theorem \ref{thm_U_cycle}.  As we mentioned in the introduction, there are two possible proofs and we give both for completeness.  A central fact, needed for both approaches, is the following non-vanishing result from \cite{cristofaro2021smooth}.

\begin{thm}
\label{thm_non_vanish}
	$\HMbar_*(Y,\frs, c_b;R)\neq 0$ for every $R$. 
\end{thm}

We now first give the (very) short but (very) indirect proof, using Heegaard Floer homology and the above non-vanishing result.

\begin{proof}[The short and indirect proof]

By the Kutluhan-Lee-Taubes isomorphism \cite{kutluhan2020hf}, we have
$$
\HMbar_*(Y,\frs,c_b;\mathbb{Z}) \cong \HFinfty(Y,\frs;\mathbb{Z}),
$$
and the isomorphism preserves the U-action. By \cite[Lemma 2.3]{ozsvath2004holomorphic}, there exists a positive integer $l$ such that 
$$
(1-U^N)^l \, \HFinfty(Y,\frs;\mathbb{Z}) = 0,
$$
therefore
$$
(1-U^N)^l \, \HMbar_*(Y,\frs,c_b;\mathbb{Z}) = 0.
$$
By the universal coefficient theorem, there is a short exact sequence
\begin{multline*}
	0\to \HMbar_*(Y,\frs,c_b;\mathbb{Z})\otimes R \to \HMbar_*(Y,\frs,c_b;R)
	\\
	\to \Tor_*(\HMbar_*(Y,\frs,c_b;\mathbb{Z}), R)\to 0.
$$
\end{multline*}
Since $(1-U^N)^l$ acts as the zero map on $\HMbar_*(Y,\frs,c_b;\mathbb{Z})\otimes R$ and $\Tor_*(\HMbar_*(Y,\frs,c_b;\mathbb{Z}), R)$, we conclude that $(1-U^N)^{2l}$ acts as the zero map on $\HMbar_*(Y,\frs,c_b;R)$. It then follows from Theorem \ref{thm_non_vanish} that there exists 
$$0\neq \sigma \in \HMbar_*(Y,\frs,c_b;R)$$ such that $U^N\sigma = \sigma$.
\end{proof}

Of course, the above proof relies on the Kutluhan-Lee-Taubes isomorphism of monopole Floer homology and Heegaard Floer homology, making a more direct proof highly desirable.  We now give that more direct proof, which will require a little more work.

\begin{proof}[The direct proof]

The proof proceeds in two steps.

{\em Step 1.  Preliminaries.}

Let $\mathbb{T} = H^1(Y;\mathbb{R})/H^1(Y;\mathbb{Z})$ be the Picard torus of $Y$. Since $c_1(\frs)$ is non-torsion, we have $\dim \mathbb{T}\ge 1$. Let $\xi_1 \in H^1(\mathbb{T};\mathbb{Z}) = \text{Hom}(H^1(Y;\mathbb{Z}),\mathbb{Z})$ be given by
\begin{equation}
\label{eqn_def_xi1}
	 a \mapsto \frac{1}{2} (a \cup c_1(\mathfrak{s})) [Y].
\end{equation}
Let $\Gamma_{\xi_1}^R$ be the local system on $\mathbb{T}$ with fiber $R[T,T^{-1}]$ where the holonomy around a loop $\gamma$ is multiplication by $T^k$ for $k = \xi_1[\gamma]$.

The following result is proved by Kronheimer and Mrowka \cite[Theorem 35.1.6]{monopolesBook}\footnote{The original statement of \cite[Theorem 35.1.6]{monopolesBook} was given for $R=\mathbb{Z}$, but the same argument applies to arbitrary coefficient rings.}.

\begin{prop}[{Kronheimer--Mrowka, \cite[Theorem 35.1.6]{monopolesBook}}]
\label{prop_KM_book}
	 The Floer homology group with balanced perturbation
$\HMbar_*(Y, \mathfrak{s}, c_b;R)$ is isomorphic to the homology of a chain complex $(C,\bar{\partial}_1+ \bar{\partial}_3),$ where 
$f$ is a Morse function on $\mathbb{T}$,
$$
C = C_*(\mathbb{T},f;\Gamma_{\xi_1}^R)
$$
is the Morse complex of $\mathbb{T}$ defined by $f$ with local coefficient $\Gamma_{\xi_1}$, the map $\bar\partial_1$ is the Morse differential of $C_*(\mathbb{T},f;\Gamma_{\xi_1}^R)$, and $\bar\partial_3$ has degree $-3$ with respect to the grading given by the Morse indices of critical points of $f$.
\end{prop}

Let $(C,\bar{\partial}_1+ \bar{\partial}_3)$ be as in Proposition \ref{prop_KM_book}.
For each $s\in \mathbb{Z}^{\ge 0}$, let $\mathcal{F}_sC$ be the $R[T,T^{-1}]$--submodule of $C$ generated by the critical points of $f$ with Morse indices no greater than $s$. Then $\mathcal{F}$ defines a filtration on the chain complex 
$(C,\bar{\partial}_1+ \bar{\partial}_3)$.  We will abuse notation and use $\mathcal{F}$ to denote the induced filtration on the homology of $(C,\bar{\partial}_1+ \bar{\partial}_3).$ Then Proposition \ref{prop_KM_book} has the following corollary (see, for example, \cite[Section 5.4]{weibel1995introduction}):
\begin{cor}
\label{cor_SS}
	The filtration $\mathcal{F}$ defines a spectral sequence $S$, where 
\begin{enumerate}
	\item the $E^2$ page\footnote{This page is called the $E^1$ page in the convention of \cite[Section 5.4]{weibel1995introduction}, and is called the $E^2$ page in the convention of \cite[Section 34]{monopolesBook}. We follow the convention of \cite{monopolesBook} here.} of $S$ is $H_*(\mathbb{T};\Gamma_{\xi_1}^R)$,
	\item $S$ converges to the associated-graded space of $H_*(C,\bar{\partial}_1+ \bar{\partial}_3)$ defined by the filtration $\mathcal{F}$. 
	\qed
\end{enumerate}   
\end{cor}

The next thing we need is the following lemma.

\begin{lem}
\label{lem_compute_H(T)}
Let $\mathbb{T}'$ be a torus with dimension $b_1(Y)-1$. Then 
	$$H_*(\mathbb{T};\Gamma_{\xi_1}^R)\cong R[T,T^{-1}]/(T^N-1)\otimes_R H_*(\mathbb{T}';R)$$
	as $R[T,T^{-1}]$--modules. 
\end{lem}

\begin{proof}[Sketch of proof]
When $R=\mathbb{Z}$, this was proved in \cite[(70)]{cristofaro2021smooth}, see also \cite[Page 688]{monopolesBook}. The computation generalizes to arbitrary rings in a verbatim way.
\end{proof}

{\em Step 2.  The proof.}

We will prove the theorem using the properties of coupled Morse homology introduced in \cite[Section 33]{monopolesBook}.

Recall that Proposition \ref{prop_KM_book} is proved in \cite{monopolesBook}  by identifying the bar-version of monopole Floer homology with a \emph{coupled Morse homology} on $\mathbb{T}$, which was defined in \cite[Section 33]{monopolesBook}.  More precisely, there is a bundle $\mathcal{H}$ of (complex) Hilbert spaces over $\mathbb{T}$ and a smooth family of self-adjoint operators $L$ on $\mathcal{H}$ such that the spectrum of $L$ at each point of $\mathbb{T}$ is discrete and is unbounded from above and below. Moreover, at each critical point of $f$, the spectrum of $L$ is disjoint from $\{0\}$ and every eigenvalue of $L$ has multiplicity 1.
The generators of the chain complex $C = C_*(\mathbb{T},f;\Gamma_{\xi_1})$ as a $\mathbb{Z}$--module are
identified with the set of pairs $(p,[\varphi])$, where $p$ is a critical point of $f$, $\varphi$ is an eigenvector of $L|_p$, and $[\varphi]$ is the image of $\varphi$ in the projectivized space of $\mathcal{H}|_p$. The operator $T$ maps a pair $(p,[\varphi])$ to $(p,[\varphi'])$, where the eigenvalue of $\varphi'$ is next to and higher than the eigenvalue of $\varphi$.

Recall that the U-map on $\HMbar_*(Y, \mathfrak{s}, c_b)$ is defined by counting \emph{reducible} flow lines with index 2 (see \cite[(25.6)]{monopolesBook}), so it can be defined on coupled Morse homology.  The U-map on coupled Morse homology is defined as follows.  
Since $\mathbb{T}$ is finite dimensional, we can choose a generic nowhere-vanishing smooth section $s$ of the dual bundle of $\mathcal{H}$. 
We then have a chain map defined by counting index-2 solutions to the flow line equations \cite[(33.10a), (33.10b)]{monopolesBook} such that the spinor $\varphi$ at time $0$ satisfies $s(\varphi|_{t=0})=0$. The U-map is its induced map on homology. See the discussion after \cite[Proposition 33.3.8]{monopolesBook} for the definition of the U-map in the torsion case.

Suppose now that there is an index-2 solution to \cite[(33.10a), (33.10b)]{monopolesBook} from a generator $(p,[\varphi])$ to a generator $(q,[\psi])$. Then \cite[(33.10a)]{monopolesBook} implies that there is a Morse gradient flow from $p$ to $q$. Since the flow lines satisfy the Morse-Smale condition, the Morse index of $p$ is greater than the Morse index of $q$ as critical points of $f$, and we have $p=q$ if they have the same Morse index. When $p=q$, the total counting of solutions is equal to 1 if the eigenvalue of $\varphi$ is higher than and next to the eigenvalue of $\psi$, and the counting is equal to zero otherwise (see \cite[(25.15)]{monopolesBook} and the equation on the bottom of \cite[Page 664]{monopolesBook}).

Recall that $\mathcal{F}$ is the filtration on $C$ defined by the Morse indices of critical points of $f$. 
 By the previous argument, for each $x\in \mathcal{F}_sC$, the action of $U$ on $x$ is equal to $T^{-1} x + y$, where $y\in \mathcal{F}_{s-1}C$. 

Now we invoke the spectral sequence $S$ from Corollary \ref{cor_SS}. 
By the discussion above, the action of $U$ on the $E^2$ page of $S$ is equal to the action of $T^{-1}$.  By Lemma \ref{lem_compute_H(T)}, we have $T^{-N} = 1$ on $H_*(\mathbb{T}; \Gamma_{\xi_1}^R)$. Therefore the action of $U^N$ is the identity map on $H_*(\mathbb{T}; \Gamma_{\xi_1}^R)$. This implies that the action of $U^N$ is the identity on every page of $S$ after the $E^2$ page. As a consequence, $U^N$ is the identity on 
$$
\frac{\mathcal{F}_s H_*(C,\bar\partial_1 +\bar\partial_3)}
{\mathcal{F}_{s-1} H_*(C,\bar\partial_1 +\bar\partial_3)}
$$ 
for each $s$. 
By Theorem \ref{thm_non_vanish}, we have 
$$H_*(C,\bar\partial_1 +\bar\partial_3)\neq 0,$$
therefore there exists 
$$0\neq \sigma\in H_*(C,\bar\partial_1 +\bar\partial_3) \cong \HMbar_*(Y,\mathfrak{s},c_b;R)$$
such that $U^N\sigma = \sigma$.
\end{proof}

\begin{rmk}
The above proof implies that
\begin{equation}
\label{eqn_quantitative_cyclic_result_HM_bar}
	(1-U^N)^{b_1(Y)+1} \, \HMbar_*(Y,\mathfrak{s},c_b;R) = 0.
\end{equation}
This is because $U^N$ acts as the identity map on 
$$
\frac{\mathcal{F}_s H_*(C,\bar\partial_1 +\bar\partial_3)}
{\mathcal{F}_{s-1} H_*(C,\bar\partial_1 +\bar\partial_3)}
$$ 
and the filtration $\mathcal{F}$ has height $b_1(Y)+1$. 

If $\HPbar_*(\phi,\Gamma;\mathbb{R})$ is the periodic Floer homology of an area-preserving map $\phi$ on a closed surface $\Sigma$ with genus $g$, we have 
$$
\langle c_1(\frs_\Gamma),[\Sigma]\rangle = 2(d-g+1)
$$
where $d$ is the degree of $\Gamma$. Therefore, $d-g+1$ is an integer multiple of  $N$. When the degree of $\Gamma$ is sufficiently large, we have
$$
	\HPbar_*(\phi,\Gamma;R)\cong \HMbar^{-*}(M_\phi,\mathfrak{s}_\Gamma,c_b;R),
$$
 therefore \eqref{eqn_quantitative_cyclic_result_HM_bar} implies 
\begin{equation}
\label{eqn_quantitative_cyclic_result_HP}
(U^{d-g+1}-1)^{b_1(M_\phi)+1}\, \HPbar_*(\phi,\Gamma;R) = 0.
\end{equation}
\end{rmk}

\subsection{Existence of non-cyclic elements}
\label{sec_non_cyclic}
This section proves Theorem \ref{not_U_cycle}.

We will follow the notation from \cite[Sections 33-35]{monopolesBook} on coupled Morse homology. 
Recall the following construction from \cite[Section 33.2]{monopolesBook}. 
For each $z\in \UU(k)$, 
let $C^\infty(S^1;z)$ be the space of smooth functions $h:\mathbb{R}\to \mathbb{C}^k$ satisfying 
$h(t + 1) = zh(t).$ Define
$$\|h\|_{L^2}^2 = \int_0^1 |h(t)|^2\,dt,$$
$$\|h\|_{L_1^vu2}^2 = \int_0^1 |h(t)|^2+|h'(t)|^2\,dt.$$
Let $H(z)$, $H_1(z)$ be the completions of $C^\infty(S^1;z)$ using the $L^2$ norm and the $L_1^2$ norm respectively. Consider 
\begin{align*}
L_{\UU(k)}(z): H_1(z) &\to H(z)\\
  h &\mapsto -i\frac{d}{dt} h.
\end{align*}
Then $H(z)$ defines a bundle $\mathcal{H}_{\UU(k)}$ of Hilbert spaces over $\UU(k)$ and $L$ is a family of self-adjoint operators on $\mathcal{H}_{\UU(k)}$. The pair $(\mathcal{H}_{\UU(k)},L_{\UU(k)})$ is called the \emph{canonical} family on $\UU(k)$.

A key ingredient in the proof of Theorem  \ref{not_U_cycle} is Proposition \ref{prop_U=T^-1_dim_le_3} below. This result was essentially also independently observed by Eismeier--Lin (in the paragraph above \cite[Proposition 6.1]{lin2021monopole}), without giving detailed proof:

\begin{prop}
\label{prop_U=T^-1_dim_le_3}
	Suppose $Q$ is a closed manifold with dimension less than or equal to $3$, and suppose $(\mathcal{H},L)$ is given by the pull-back of the canonical family on $\UU(2)$. Then after homotoping the family $(\mathcal{H},L)$ on $Q$ and with a suitable choice of Morse function $f$, we have $U=T^{-1}$ on the coupled Morse homology $\bar H_*(Q,L)$. 
\end{prop}

\begin{proof}
Consider the canonical family $\mathcal{H}_{\UU(1)}$ over $\UU(1)$. Let $1\in\UU(1)$ denote the identity element. 
For each $h\in C^\infty(S^1;1)$, the map 
$$
r \mapsto (t \mapsto e^{irt} h(t))
$$
defines a lifting of 
\begin{align*}
\mathbb{R} &\to \UU(1) \\ 
r &\mapsto e^{ir}.
\end{align*}
There exists a unique connection $\nabla$ on $\mathcal{H}_{\UU(1)}$ such that the above lifting is flat for all $h$.  
Locally, the bundle $\mathcal{H}_{\UU(1)}$ decomposes as a direct sum of eigenspaces of $L_{\UU(1)}(z)$, and the parallel translations of $\nabla$ preserve this decomposition.

Suppose $\gamma:\mathbb{R}\to\UU(1)$ is a path on $\UU(1)$ such that 
$$\lim_{s\to -\infty} \gamma(s) = a\in \UU(1)$$ 
and
$$\lim_{s\to +\infty} \gamma(s) = b\in \UU(1).$$ 
Also assume that $a,b\neq 1$ so that $0$ is not an eigenvalue of $L_{\UU(1)}(a)$ or $L_{\UU(1)}(b)$. Let $p(a)$ be an eigenvector of $L_{\UU(1)}(a)$ and $p(b)$ be an eigenvector of $L_{\UU(1)}(b)$.  By \cite[Proposition 33.3.2]{monopolesBook}, the moduli space of solutions to \cite[(33.9(b))]{monopolesBook} from $[p(a)]$ to $[p(b)]$ over $\gamma$ has dimension zero if and only if $[p(b)]$ is the image of $[p(a)]$ under the parallel translation of $\nabla$ with respect to $\gamma$. In this case, if $\phi$ is a section of $\gamma^*(\mathcal{H}_{\UU(1)})$ from $[p(a)]$ to $[p(b)]$ that solves \cite[(33.9(b))]{monopolesBook}, then $\phi/|\phi|$ is flat with respect to $\gamma^*\nabla$ and $\phi(s)$ is an eigenvector of $L_{\UU(1)}(\gamma(s))$ for every $s\in \mathbb{R}$.

Now let $Q$ be the manifold in the statement of the proposition. By the assumptions, $Q$ is closed and $\dim Q\le 3$. Let $(\mathcal{H}_{\UU(2)},L_{\UU(2)})$ be the canonical family over $\UU(2)$.   Let
$$v:Q\to \UU(2)$$
 be a smooth map
such that the pair $(\mathcal{H},L)$ on $Q$ is given by the pull-back of $(\mathcal{H}_{\UU(2)},L_{\UU(2)})$ via $v$. We study the coupled Morse homology of $(\mathcal{H},L)$.

By \cite[Proposition 33.2.2]{monopolesBook}, the restriction of $(\mathcal{H}_{\UU(2)},L_{\UU(2)})$ to $\UU(1)$ is homotopic to $(\mathcal{H}_{\UU(1)},L_{\UU(1)})$. 
Let $N(\UU(1))\subset \UU(2)$ be a tubular neighborhood of $\UU(1)$. Then $N(\UU(1))\cong D^2\times \UU(1)$. Let $\pi: N(\UU(1))\to \UU(1)$ be the projection map. 
Homotope $(\mathcal{H}_{\UU(2)},L_{\UU(2)})$ on $\UU(2)$ so that 
\begin{equation}
\label{eqn_homotope_canonical_family_U(2)}
	(\mathcal{H}_{\UU(2)},L_{\UU(2)})|_{N(\UU(1))} = (\pi^*(\mathcal{H}_{\UU(1)})\oplus L_0,\pi^*(L_{\UU(1)})\oplus L_0)
\end{equation}
where $L_0$ is a constant self-adjoint operator on some fixed Hilbert space $H_0$. Here we abuse notation and use $(\mathcal{H}_{\UU(2)},L_{\UU(2)})$ to also denote the pair after homotopy. Although \cite[Proposition 33.2.2]{monopolesBook} already implies that \eqref{eqn_homotope_canonical_family_U(2)} can be achieved with $(H_0,L_0)$ being zero, we will give an argument that is valid for arbitrary $(H_0,L_0)$. This flexibility will be needed later (see Remark \ref{rmk_homotopy_coupled_Morse} below).

Let $\mathcal{H}_{\UU(1)}^*$ be the dual bundle of $\mathcal{H}_{\UU(1)}$. 
Let $s_0$ be a section of 
$$\mathcal{H}_{\UU(1)}^*\oplus H_0^*$$
 such that $s(e)\neq 0$ whenever $e$ is a non-zero eigenvector of $L_{\UU(1)}\oplus L_0$. This condition holds for generic $s$ because $s(e)\neq 0$ is a codimension 2 condition and $\UU(1)$ has dimension 1. 
Let $s$ be a section of the dual bundle of $\mathcal{H}_{\UU(2)}$ such that $s=\pi^*(s_0)$ on $N(\UU(1))$.

Let $f$ be a fixed Morse function on $Q$. 
Notice that $\UU(2)$ is diffeomorphic to $S^3\times S^1$, where $\UU(1)\subset \UU(2)$ is isotopic to $\{\pt\}\times S^1$.  Therefore, the complement of $N(\UU(1))$ in $\UU(2)$ retracts to $S^1$. 
Since $S^1$ has codimension $3$ in $U^2$, 
 we can homotope the map $v$ so that if $a,b$ are two critical points of $f$ on $Q$ such that $\ind a = \ind b+2$, then $v(M(a,b))\subset N(\UU(1))$. Here, $M(a,b)\subset Q$ denotes the locus of negative gradient flow lines of $f$ from $a$ to $b$.

 Now we can compute the $U$ map on the coupled Morse homology $\bar H_*(Q,L)$. Since $\dim Q\le 3$, the $U$ map is decomposed to two parts.  The first part comes from the flow lines that project to the constant map on $Q$.  This part contributes to a term $T^{-1}$.  The second part comes from the flow lines whose projections to $Q$ are Morse flow lines on $Q$ from a critical point $a$ to a critical point $b$ such that $\ind a = \ind b+2$.  The contribution of the second part is zero since we always have $s(\phi(0))\neq 0$ with the above choice of $s$ for every solution of \cite[(33.9(b))]{monopolesBook}. In conclusion, we have $U=T^{-1}$.
 \end{proof}
 
 \begin{proof}[Proof of Theorem \ref{not_U_cycle}]
 Since $b_1(Y)=3$ and $H^1(Y;\mathbb{Z})$ has no torsion, we have $H^1(Y;\mathbb{Z})\cong \mathbb{Z}^3$. Let $a_1,a_2,a_3$ be a basis for $H^1(Y;\mathbb{Z})$, and suppose 
 $\langle a_1\cup a_2\cup a_3, [Y]\rangle = m$. By the assumptions, we have $m\neq 0$. 

By \cite[Equation (35.5)]{monopolesBook}, the Floer homology group 
$$\HMbar^*(Y,\frs,c_b;\mathbb{Z})\cong \HMbar_*(-Y,\frs,c_b;\mathbb{Z})$$ 
is given by the co-kernel of a map from $\mathbb{Z}^4\otimes\mathbb{Z}[T^{-1},T]$ to $\mathbb{Z}^4\otimes\mathbb{Z}[T^{-1},T]$ defined by the matrix
\begin{equation}
\label{eqn_b1(Y)=3_differential}
\begin{pmatrix}
T^N-1 & 0 & 0 & -mT \\
0 & T^N-1 & 0 & 0 \\
0 & 0 & T^N-1 & 0 \\
0 & 0 & 0 & T^N-1	
\end{pmatrix}.
\end{equation}
The extra negative sign in front of the term $mT$ comes from the fact that we are computing the Floer \emph{cohomology} of $Y$ instead of \emph{homology}.  By Proposition \ref{prop_U=T^-1_dim_le_3}, the action of $U$ is equal to the multiplication by $T^{-1}$. 

For each non-zero integer $i$, let $\sigma_i$ be the element represented by the column vector $(0,0,0,i)^t$ in the co-kernel of \eqref{eqn_b1(Y)=3_differential}. Then we have $n(T^N-1)\sigma_i \neq 0$ for every non-zero integer $n$, and $(T^N-1)^2\sigma_i = 0$. 

Suppose $(T^k-1)\sigma_i=0$ for some non-zero integer $k$. Then we have $(T^{|k|}-1)\sigma_i=0$. Let $\ell = gcd(N,|k|)>0$. Then $T^{\ell}-1$ is the greatest common divisor of $(T^N-1)^2$ and $T^{|k|}-1$ in 
 $\mathbb{Q}[T]$. Since $\mathbb{Q}[T]$ is a PID, there exist $u,v\in \mathbb{Q}$ such that
$$
u (T^N-1)^2 + v (T^{|k|}-1) = T^\ell -1.
$$
So there exist $u',v'\in \mathbb{Z}$ and a non-zero integer $n$ such that
$$
u' (T^N-1)^2 + v' (T^{|k|}-1) = n(T^\ell -1).
$$
As a consequence,
$$
n(T^\ell -1)\sigma_i = u' (T^N-1)^2\sigma_i + v' (T^{|k|}-1)\sigma_i = 0.
$$
Since $T^N-1$ is a multiple of $T^\ell-1$ in $\mathbb{Z}[T]$, this implies $n(T^N -1)\sigma_i=0$, which yields a contradiction. 

In conclusion, $(T^{k}-1)\sigma_i\neq 0$ for every non-zero integer $k$.
It is straightforward to verify that $\sigma_i\neq \sigma_j$ in the co-kernel of \eqref{eqn_b1(Y)=3_differential} whenever $i\neq j$. Therefore the desired result is proved.
 \end{proof}

 \begin{rmk}
 \label{rmk_homotopy_coupled_Morse}
 	 In \cite{monopolesBook}, it is proved that 
 	 $\bar H_*(Q,L)$ 
 is invariant under homotopies as a $\mathbb{Z}[U]$--module, but it is not proved that $\bar H_*(Q,L)$ is invariant as a $\mathbb{Z}[T]$--module. Our proof of Proposition \ref{prop_U=T^-1_dim_le_3} only shows that  $U=T^{-1}$ after homotoping $v$ and the canonical family over $\UU(2)$, and with a suitable choice of the Morse function $f$ and the connection $\nabla$. The computation of \eqref{eqn_b1(Y)=3_differential} is a consequence of \cite[Proposition 34.4.1]{monopolesBook} (see Proposition \ref{prop_KM_book} above), whose proof also requires homotoping $v$ and the canonical family on $\UU(2)$, and a choice of the Morse function. A closer look at the proof shows that the homotopies of $v$ and the canonical family on $\UU(2)$ used in \cite[Proposition 34.4.1]{monopolesBook} satisfy all the conditions we needed for Proposition \ref{prop_U=T^-1_dim_le_3} in a verbatim way (the neighborhood $N(\UU(1))$ in the proof of Proposition \ref{prop_U=T^-1_dim_le_3} corresponds to the complement of $S^1\times W$ in the proof of \cite[Proposition 34.4.1]{monopolesBook}). One can choose $f$ so that 
 \begin{enumerate}
 	\item If $a,b$ are two critical points of $f$ on $Q$ such that $\ind a = \ind b+2$, then $v(M(a,b))\subset N(\UU(1))$,
 	\item $f$ equals the pull-back of the projection from 
 	 $$\UU(2)\backslash N(\UU(1))\cong [0,1]\to S^3$$ via $v$ on the pre-image of $\UU(2)\backslash N(\UU(1))$.
 \end{enumerate}
In this case, the Morse function $f$ satisfies the conditions for both \cite[Proposition 34.4.1]{monopolesBook} and Proposition \ref{prop_U=T^-1_dim_le_3}.  Therefore, Proposition \ref{prop_U=T^-1_dim_le_3} indeed implies $U=T^{-1}$ if we identify $\HMbar^*(Y,\frs,c_b;\mathbb{Z})$ with the co-kernel of \eqref{eqn_b1(Y)=3_differential}.
 \end{rmk}
 
\bibliographystyle{alpha}
\bibliography{references}

\end{document}